\documentclass[12pt, reqno]{amsart}
\usepackage{amsmath, amsthm, amscd, amsfonts, amssymb, graphicx, color}

\newtheorem{theorem}{Theorem}[section]
\newtheorem{lemma}[theorem]{Lemma}

\newtheorem{corollary}[theorem]{Corollary}
\theoremstyle{definition}
\newtheorem{definition}[theorem]{Definition}
\newtheorem{example}[theorem]{Example}

\newtheorem{question}[theorem]{Question}

\theoremstyle{remark}
\newtheorem{remark}[theorem]{Remark}
\begin{document}

\setcounter{page}{1}

\title{New topologies derived from the old one via ideals
}

\author{F.Y. Issaka$^{\rm 1}$, M. Özkoç$^{\rm 2,\ast}$}

\address{$^{1}$Department of Mathematics, Graduate School of Natural and Applied Sciences, Muğla Sıtkı Koçman University, 48000, Menteşe-Muğla, Turkey.}
\email{faicalyacine@gmail.com}

\address{$^{2}$Department of Mathematics, Faculty of Science, Muğla Sıtkı Koçman University, 48000, Menteşe-Muğla, Turkey.}
\email{murad.ozkoc@mu.edu.tr \& murad.ozkoc@gmail.com}




\date{$^{*}$Corresponding author}

\begin{abstract}
The main purpose of this paper is to introduce and study minimal and maximal ideals defined on ideal topological spaces. Also, we define and investigate the concepts of ideal quotient and annihilator of any subfamily of $2^X$, where $2^X$ is the power set of $X.$ We obtain some of their fundamental properties. In addition, several relationships among the above notions have been discussed. Moreover, we get a new topology, called sharp topology via the sharp operator defined in the scope of this study, finer than the old one. Furthermore, a decomposition of the notion of open set has been obtained. Finally, we conclude our work with some interesting applications. 
\end{abstract} 
\maketitle

\section{Introduction}
Some classical structures such as filters \cite{filtre}, ideals \cite{c}, grills\cite{grill}, and also primals \cite{primal} are some of the topics hard studied in the area of general topology. An ideal $\mathcal{I}$ on a topological space $(X,\tau)$ is a non-empty collection of subsets of $X$ which satisfies (i) $A\in \mathcal{I}$ and $B \subseteq A$ implies $B\in \mathcal{I}$ and (ii) $A\in \mathcal{I}$ and $B\in \mathcal{I}$ implies $A\cup B\in \mathcal{I}$. A topological space with an ideal is called ideal topological space. The concept of the local function in general topology was introduced by Kuratovski \cite{kur} in 1933 and studied from very different aspects by many mathematicians. In an ideal topological space $\left(X,\tau,\mathcal{I}\right)$, the local function $(\cdot)^{*}$ \cite{kur} is defined as $A^{*}\left(\mathcal{I},\tau\right)=\{ x\in X|\left(\forall U\in \tau\left(x\right)\right)\left(U\cap A\notin\mathcal{I}\right)\},$ where $\tau\left(x\right)$ is the collection of all open subsets containing $x\in X$.  Especially,  Vaidyanathaswamy \cite{vai} investigated more detailed properties of the local function in 1945. Thanks to the concept of the local function, the literature gained a new topology called $*$-topology and was studied further by Hayashi \cite{h} in 1964 and Njastad \cite{nja} in 1966, later by Samuel \cite{g} in 1975 and many others.
In 1990, after a hiatus of about 15 years, this topic was revisited by Jankovic and Hamlett \cite{jan}. In that article, they have not only summarised all the known facts on this topic, but also presented some new results.

Subsequently, many papers have been published on this topic. For instance, Arenas et al. \cite{are} studied the idealization of some weak separation axioms, while Navaneethakrishnan \cite{nav} in 2008 devoted their attention to investigating $g$-closed sets in ideal topological spaces. The others such as Hatır \cite{hat} and Ekici \cite{eki} have studied the decompositions of continuity in ideal topological spaces and  $I$-Alexandroff topological spaces in ideal topological spaces, respectively.

Most of these approaches to the subject are not very different with each other. In one approach, a new topology with new properties is obtained by changing the definition of the local function, while in another approach, topologies arising from more or less known different ideals are considered together.


In section 3, we define the notions of the maximal and minimal ideals. Some characterizations of these concepts are obtained.  

In section 4, we introduce and study the notion of ideal quotient and investigate some of its fundamental properties. We also give a characterization for the notion of the maximal ideal through ideal quotient. 

In section 5, we define the concepts of the annihilator of a set family and faithful ideal. We obtain a characterization the concept of the minimal ideal with the help of the concept of annihilator. We also give a new characterization of denseness via annihilator.

In section 6, we introduce a new operator called sharp operator and obtain some of its fundamental properties. Also, we create a new Kuratowski closure operator through the sharp operator. The topology obtained via this Kuratowski closure operator come across finer than the original one. Moreover, we reveal a decomposition of an open set.

In section 7, we introduce the concepts of $*$-continuity and $\sharp$-continuity. We give a relation  between continuity and $\sharp$-continuity. Furthermore, we obtain a new decomposition of continuity.

In the last section, we give some applications of sharp operator and prove the denseness of the set of all rational numbers using the notions of sharp topology and annihilator defined in the scope of this paper. We are also looking for answers to the following questions:

\begin{itemize}
\item Is there any Hausdorff space on $\mathbb{R}$ such that the set of all irrational numbers $\mathbb{I}$ is not dense, while the set of all rational numbers $\mathbb{Q}$ is dense?

\item Is there any Hausorff space on $\mathbb{R}$ such that the set of all rational numbers $\mathbb{Q}$ is clopen?

\item Is there a disconnected Hausdorff topological space on $\mathbb{R}$ except discrete topological space?
\end{itemize}
\section{Preliminaries}
Throughout this paper, $(X,\tau)$ and $(Y,\sigma)$ (or simply $X$ and $Y$) always mean topological spaces on which no separation axioms are assumed unless otherwise stated. For a subset $A$ of a space $X$, the closure and the interior of $A$ will be denoted by $cl(A)$ and $int(A),$ respectively. 
The operator $cl^{*}:2^X\to 2^X$ defined by $cl^*(A)=A\cup A^*$ is a Kuratowski closure operator. The topology induced by the operator $cl^*$ is $\tau^{*}(\mathcal{I},\tau)=\{A\subseteq X|cl^*(X\setminus A)=X\setminus A\}$ and called $*$-topology which is finer than $\tau.$ Natkaniec \cite{d} have introduced the complement of local function called $\Psi$-operator which is defined by $\Psi(A)=X\setminus(X\setminus A)^*$ for any subset $A$ of $X.$ An ideal topological space $\left(X,\tau,\mathcal{I} \right)$ is called a Hayashi-Samuel space \cite{hays} if $\tau\cap \mathcal{I}=\{\emptyset\}$. 

\begin{lemma}\label{gulkız}
Let $(X,\tau,\mathcal{I})$ be an ideal topological space and $A\subseteq X.$ If $A\in \mathcal{I},$ then $A^*=\emptyset.$
\end{lemma}

\begin{lemma} \cite{ser}\label{8}
Let $\left(X,\tau \right ) $ be a topological space and $\mathcal{I},\mathcal{J}\subseteq 2^X$ be ideals on $X$. Then, the following hold for any subset $A\subseteq X.$
\\

a) $A^*(\mathcal{I}\cap\mathcal{J},\tau ) = A^*(\mathcal{I},\tau )\cup A^*(\mathcal{J},\tau ), $
\\

b) $\tau^*(\mathcal{I}\cap\mathcal{J},\tau ) = \tau^*(\mathcal{I},\tau )\cap \tau^*(\mathcal{J},\tau ). $ 
\end{lemma}
\begin{definition}
Let $X$ be a non-empty set and $A\subseteq X.$ Then, the ideal generated by $A$ is defined as $\mathcal{I}(A):=\{I | I\subseteq A\}.$
\end{definition}

\begin{lemma}
Let $X$ be a non-empty  set and $A\subseteq X.$ Then, the family $\mathcal{I}_{\epsilon}(A)=\{I\subseteq X | I\cap A=\emptyset\}$ is an  ideal on $X.$
\end{lemma}

\section{Maximal and minimal ideal on topological space}

\begin{definition}
Let $\mathcal{I}$ be a proper ideal on $X$ i.e. $\mathcal{I}\neq 2^X.$ Then, $\mathcal{I}$ is said to be a maximal ideal if for any ideal $\mathcal{J}$ with $\mathcal{I}\subseteq \mathcal{J}$, $\mathcal{I}= \mathcal{J}$ or $ \mathcal{J} = 2^X$.
\end{definition}
\begin{theorem}\label{1}
Let $\mathcal{I}$ be a proper ideal on $X.$ Then, $\mathcal{I}$ is a maximal ideal if and only if for all $A,B\subseteq X,$ if $A\cap B \in \mathcal{I}$ then $A\in \mathcal{I}$ or $B\in \mathcal{I}$.
\end{theorem}
\begin{proof}
$(\Rightarrow):$ Let $\mathcal{I}$ be a maximal ideal and $A\cap B\in \mathcal{I}.$
\\
\\
$\left. \begin{array}{rr} 
\mathcal{J}:=\{A\cap B | A\in \mathcal{I} \vee B\in \mathcal{I}\} \Rightarrow (\mathcal{J} \text{ is an ideal})(\mathcal{I}\subseteq \mathcal{J}) \\ \mathcal{I} \text{ is maximal ideal} \end{array}	\right\}\Rightarrow$
\\
\\
$\left. \begin{array}{rr} 
\Rightarrow \mathcal{I}=\mathcal{J} \\ A\cap B \in \mathcal{I}\end{array}	\right\} \Rightarrow A\cap B \in\mathcal{J}\Rightarrow A\in \mathcal{I} \vee B\in \mathcal{I}.$
\\

$(\Leftarrow):$ Let $\mathcal{J}$ be an ideal such that $\mathcal{J}\neq 2^X$ and $\mathcal{I}\subseteq\mathcal{J}.$ We will prove that $\mathcal{I}=\mathcal{J}.$ Let $A\in \mathcal{J}.$
\\
$\left. \begin{array}{rr} 
  \mathcal{I}\text{ is an ideal}\Rightarrow \emptyset \in \mathcal{I} \\ \emptyset= A\cap (X\setminus A) \end{array}	\right\}\Rightarrow \!\!\!\!\! \begin{array}{c} \\
	\left. \begin{array}{r} 
	A\cap (X\setminus A) \in \mathcal{I} \\ \text{Hypothesis}  \end{array}	\right\} \Rightarrow  A\in \mathcal{I} \vee (X\setminus A) \in \mathcal{I} 
\end{array}$  
\\
$\left. \begin{array}{rr} 
  \Rightarrow A\in \mathcal{I} \vee (X\setminus A) \in \mathcal{I} \\ \mathcal{I}\subseteq\mathcal{J} \end{array}	\right\}\Rightarrow \!\!\!\!\! \begin{array}{c} \\
	\left. \begin{array}{r} 
 A\in \mathcal{I} \vee (X\setminus A) \in \mathcal{J}  \\  A\in \mathcal{J}  \end{array}	\right\} \Rightarrow  
\end{array}$
\\ 
\\
$\left. \begin{array}{rr} 
\Rightarrow X=A\cup (X\setminus A) \in \mathcal{J} \vee A\in \mathcal{I} \\ \mathcal{J}\neq 2^X \end{array}\right\} \Rightarrow A\in \mathcal{I}$
\\

Thus, we have $\mathcal{J}\subseteq \mathcal{I}.$ Since $\mathcal{I}\subseteq  \mathcal{J},$ we  get $\mathcal{I}= \mathcal{J}.$ 
\end{proof}

\begin{theorem} \label{maxkar}
Let $\mathcal{I}$ be a proper ideal on $X.$ Then, $\mathcal{I}$ is a maximal ideal if and only if  $A\in \mathcal{I}$ or $X\setminus A \in \mathcal{I}$ for all $A\subseteq X.$
\end{theorem}
\begin{proof}
 $(\Rightarrow):$ Let $\mathcal{I}$ be a maximal ideal and $A\subseteq X.$
\\
$\left. \begin{array}{rr} 
  \mathcal{I}\text{ is an ideal}\Rightarrow \emptyset \in \mathcal{I} \\ A\subseteq X\Rightarrow \emptyset=A\cap (X\setminus A) \end{array}	\right\}\Rightarrow \!\!\!\!\!\!\!\!\!\! \begin{array}{c} \\
	\left. \begin{array}{r} 
	A\cap (X\setminus A) \in \mathcal{I} \\  \mathcal{I}\text{ is maximal ideal}  \end{array}	\right\} \overset{\text{Theorem }\ref{1}}{\Rightarrow}  A\in \mathcal{I} \vee (X\setminus A) \in \mathcal{I}.
\end{array}$
\\

$(\Leftarrow):$ Let $\mathcal{I}$ and $\mathcal{J}$ be two ideals such that $\mathcal{J}\neq 2^X$ and $\mathcal{I}\subseteq \mathcal{J}.$ We will prove that $\mathcal{I}=\mathcal{J}.$ Now, let $A\in \mathcal{J}.$
\\
$\left. \begin{array}{rr} 
  A\in\mathcal{J}\Rightarrow A\subseteq X \\ \text{Hypothesis} \end{array}	\right\}\Rightarrow \!\!\!\!\! \begin{array}{c} \\
	\left. \begin{array}{r} 
 A\in \mathcal{I} \vee (X\setminus A) \in \mathcal{I}  \\  \mathcal{I}\subseteq \mathcal{J}  \end{array}	\right\} \Rightarrow  
\end{array}$ 
\\
$\left. \begin{array}{rr} 
  \Rightarrow  A\in \mathcal{I} \vee (X\setminus A) \in \mathcal{J}  \\  A\in \mathcal{J} \end{array}	\right\}\Rightarrow \!\!\!\!\! \begin{array}{c} \\
	\left. \begin{array}{r} 
X= A\cup (X\setminus A)\in \mathcal{J} \vee A \in \mathcal{I}  \\  \mathcal{J}\neq 2^X  \end{array}	\right\} \Rightarrow A\in \mathcal{I}
\end{array}$
\\

Thus, we have $\mathcal{J}\subseteq \mathcal{I}.$ Since  $\mathcal{I}\subseteq\mathcal{J},$  we  get $\mathcal{I}= \mathcal{J}.$ 
\end{proof}

\begin{corollary}
 Let $(X,\tau,\mathcal{I})$ be an ideal topological space. If $\mathcal{I}$ is a maximal ideal, then  $A^*=\emptyset$ or $(X\setminus A)^*=\emptyset$ for all $A\subseteq X.$
\end{corollary}
\begin{theorem}\label{22}
Let $(X,\tau,\mathcal{I})$ be an ideal topological space and $A\subseteq X.$ If $\mathcal{I}$ is a maximal ideal, then $A$ is $\tau^*$-closed or $\tau^*$-open.
\end{theorem}
\begin{proof}
 Let $\mathcal{I}$ be a maximal ideal on $X$ and $A\subseteq X$.
$$\begin{array}{rcl} (\mathcal{I} \text{ is maximal ideal})(A\subseteq X)& \Rightarrow & A\in \mathcal{I} \vee (X\setminus A ) \in \mathcal{I}\\ & \Rightarrow &  A^*=\emptyset \vee (X\setminus A )^* = \emptyset \\ & \Rightarrow & A^*\cup A =A \vee \Psi( A )=X\setminus(X\setminus A )^* = X  \\ & \Rightarrow & cl^*(A)=A \vee A\subseteq \Psi( A ) \\ & \Rightarrow & A\in C(X,\tau^*) \vee A\in \tau^* . \qedhere
\end{array}$$
\end{proof}
\begin{corollary}
Let $(X,\tau,\mathcal{I})$ be an ideal topological space. If $\mathcal{I}$ is a maximal ideal, then $(X,\tau^*)$ is a $T_0$ space.
\end{corollary}

\begin{proof}
This follows from Theorem \ref{22}.
\end{proof}

\begin{definition}
Let $\mathcal{I}$ be a proper ideal on $X$ such that $\mathcal{I}\neq \{\emptyset\}.$ Then, $\mathcal{I}$ is said to be a minimal ideal if for any ideal $\mathcal{J}$ with $\mathcal{J}\subseteq \mathcal{I}$, $\mathcal{I}= \mathcal{J}$ or $ \mathcal{J} = \{\emptyset\}$.
\end{definition}
\begin{theorem}
Let $\mathcal{I}$ be a proper ideal on $X.$ Then, $\mathcal{I}$ is a minimal ideal if and only if $A=B$ for all $A,B\in\mathcal{I}\setminus \{\emptyset\}.$  
\end{theorem}
\begin{proof}
 $(\Rightarrow): $ Let $\mathcal{I} $ be a minimal ideal and $A,B\in \mathcal{I}\setminus \{\emptyset\}.$
 \\
 $
 \left. \begin{array}{r} 
 A,B\in \mathcal{I}\setminus \{\emptyset\}\Rightarrow (\mathcal{I}(A)\subseteq \mathcal{I})(\mathcal{I}(B)\subseteq \mathcal{I})  \\  \mathcal{I} \text{ is minimal}  \end{array}	\right\} \Rightarrow \mathcal{I}(A)=\mathcal{I}(B)=\mathcal{I} \Rightarrow A=B. $
 \\
 
$(\Leftarrow):$ Let $\mathcal{I}$ be an ideal and $\mathcal{J}\subseteq \mathcal{I}.$\\
$\left. \begin{array}{rr} 
 \mathcal{I} \text{ is an ideal } \\ \text{ Hypothesis } \end{array}	\right\}\Rightarrow \!\!\!\!\! \begin{array}{c} \\
\left. \begin{array}{r} 
 |\mathcal{I}| = 2  \\  \mathcal{J}\subseteq \mathcal{I} \end{array}	\right\} \Rightarrow (\mathcal{J}|= 1 \vee |\mathcal{J}|=2) \Rightarrow (\mathcal{J}=\{\emptyset\} \vee \mathcal{J}= \mathcal{I}). 
\end{array}$ 
\end{proof}
\begin{corollary}
Let $\mathcal{I}$ be a proper ideal on $X$. Then, the following statements are equivalent:
\\

$1)$ $\mathcal{I} $ is minimal ideal on $X;$\\

$2)$ $|\mathcal{I}|=2,$ where $|\mathcal{I}|$ is the cardinality of $\mathcal{I}.$ \\

$3)$ There exists a subset $A$ of $X$ such that $|A|=1$ and $\mathcal{I}=\mathcal{I}(A).$
\end{corollary}
\begin{theorem}\label{17}
Let $X$ be a non-empty set and $A\subseteq X.$ Then, $\mathcal{I}(A)$ is a minimal ideal on $X$ if and only if $\mathcal{I}_\epsilon(A)$ is a maximal ideal on $X$.
\end{theorem}
\begin{proof}
$(\Rightarrow):$ Let $\mathcal{I}(A)$ be a minimal ideal on $X.$
\\
$\begin{array}{rcl}
\left. \begin{array}{rr} 
  \mathcal{I}(A)\text{ is minimal} \Rightarrow (\exists x\in X)(A=\{x\})(\mathcal{I}=\mathcal{I}(A))  \\ B\subseteq X  \end{array}	\right\}  \Rightarrow  
\end{array}$
\\
$\begin{array}{rr}
\Rightarrow A\cap B=\emptyset \vee A\cap (X\setminus B)=\emptyset\end{array} $
\\
$\begin{array}{l}
\Rightarrow B\in \mathcal{I}_\epsilon(A) \vee X\setminus B\in \mathcal{I}_\epsilon(A) 
\end{array}$

This means that $\mathcal{I}_\epsilon(A)$ is maximal ideal due to Theorem \ref{maxkar}.
\\
  
$(\Leftarrow):$ Let $\mathcal{I}_\epsilon(A)$ be a maximal ideal and $x,y\in A .$
\\
$\begin{array}{rcl}
\left. \begin{array}{rr} 
x,y\in A \Rightarrow (\mathcal{I}_\epsilon(A)\subseteq \mathcal{I}_\epsilon(\{x\}))(\mathcal{I}_\epsilon(A) \subseteq \mathcal{I}_\epsilon(\{y\}))\\ \mathcal{I}_\epsilon(A) \text{ is maximal}   \end{array}	\right\}   \Rightarrow   \end{array}$
\\
$
\begin{array}{l}
\Rightarrow \mathcal{I}_\epsilon(\{x\})=\mathcal{I}_\epsilon(\{y\})=\mathcal{I}_\epsilon(A) \\ 
\end{array}
$
\\
$
\left.\begin{array}{l}
\Rightarrow  x=y \\ x,y\in A \end{array}\right\}\Rightarrow |A|=1\Rightarrow   \mathcal{I}(A) \text{ is minimal}.
$
\end{proof}

\begin{corollary}
Let $\mathcal{I}$ be a proper ideal on $X$. Then, $\mathcal{I}$ is a maximal ideal if and only if there exists a singleton subset $A\subseteq X$ such that $\mathcal{I}=\mathcal{I}_\epsilon(A)$.
\end{corollary}
\section{Ideal Quotient}
\begin{lemma}
Let $\mathcal{I}$ be an ideal on $X$ and $\mathcal{J} \subseteq 2^X.$ Then, the family  $(\mathcal{I}:\mathcal{J})=\{ A \subseteq X| (\forall J \in  \mathcal{J}) (A\cap J \in\mathcal{I})\}$ is an ideal on $X$.
 
\end{lemma}
 \begin{proof} 
Let $A\in (\mathcal{I}:\mathcal{J})$ and $B\subseteq A.$ We will prove that $(\mathcal{I}:\mathcal{J})$ is downward closed.
\\ 
$\left.\begin{array}{rr} A\in (\mathcal{I}:\mathcal{J}) \Rightarrow (\forall J \in \mathcal{J} ) (A\cap J \in \mathcal{I}) \\ B\subseteq A \end{array} \right\}\Rightarrow\!\!\!\!\!\! \begin{array}{rr} \\  \left. \begin{array}{rr} (\forall J \in \mathcal{J} ) (B\cap J \subseteq A\cap J \in \mathcal{I}) \\ \mathcal{I} \text{ is an ideal} \end{array} \right\} \Rightarrow \end{array}$
\\
\\
$
\begin{array}{l}
\Rightarrow (\forall J \in \mathcal{J} ) (B\cap J \in \mathcal{I})
\end{array}
$
\\
$
\begin{array}{l}
\Rightarrow B \in (\mathcal{I}:\mathcal{J}).
\end{array}
$
\\

Now, let $A,B \in  (\mathcal{I}:\mathcal{J}).$ We will prove that $A\cap B \in (\mathcal{I}:\mathcal{J}).$
\\
$
\begin{array}{rr} 
A,B\in (\mathcal{I}:\mathcal{J}) \Rightarrow (\forall J \in \mathcal{J})(A\cap J \in \mathcal{I})(\forall J \in \mathcal{J})(B\cap J \in \mathcal{I}) 
\end{array}
$
\\
$\left. \begin{array}{rr} 
\Rightarrow (\forall J \in \mathcal{J})(A\cap J \in \mathcal{I})(B\cap J \in \mathcal{I}) \\ \mathcal{I} \text{ is an ideal} \end{array}	\right\}\Rightarrow (\forall J\in\mathcal{J})((A\cap J)\cup (B\cap J) \in \mathcal{I})$
\\
$\begin{array}{l}
\Rightarrow (\forall J\in\mathcal{J})((A\cup B)\cap J=(A\cap J)\cup (B\cap J) \in \mathcal{I})
\end{array}$
\\
$
\begin{array}{l}
\Rightarrow A\cup B \in (\mathcal{I}:\mathcal{J}).
\end{array}
$
  \end{proof}
 \begin{definition}
 Let $\mathcal{I}$ be an ideal on $X$ and $\mathcal{J} \subseteq 2^X.$ Then, the family $(\mathcal{I}:\mathcal{J})$ is called ideal quotient.
\end{definition}

 \begin{theorem} \label{1}
Let $\mathcal{I}$ and $\mathcal{I}'$ be two ideals on $X$ and $\mathcal{J} ,\mathcal{J}'\subseteq 2^X.$ Then, the following properties hold:\\

$a)$ $\mathcal{I}\subseteq (\mathcal{I}:\mathcal{J});$\\
 
$b)$ $\mathcal{J}\subseteq \mathcal{I}$ if and only if $(\mathcal{I}:\mathcal{J}) = 2^X;$\\

$c)$ if $\mathcal{J}\subseteq \mathcal{J}',$ then $(\mathcal{I}:\mathcal{J}') \subseteq (\mathcal{I}:\mathcal{J});$\\
 
$d)$ if $X\in \mathcal{J},$ then $(\mathcal{I}:\mathcal{J}) = \mathcal{I};$\\
 
$e)$ $(\mathcal{I}\cap \mathcal{I}':\mathcal{J}) = (\mathcal{I}:\mathcal{J})\cap(\mathcal{I}':\mathcal{J}).$ 
 \end{theorem}
 
\begin{proof}
$a)$ Let $A\in \mathcal{I}.$ We will show that $A\in (\mathcal{I}:\mathcal{J})$.
$$\left. \begin{array}{rr} 
A\in \mathcal{I}\Rightarrow (\forall J\in\mathcal{J})(A\cap J\subseteq A\in\mathcal{I})  \\ \mathcal{I}\text{ is ideal} \end{array}	\right\}\Rightarrow (\forall J\in\mathcal{J})(A\cap J \in \mathcal{I})\Rightarrow A\in (\mathcal{I}:\mathcal{J}).$$
 
$b)$ $(\Rightarrow):$ Let $\mathcal{J}\subseteq \mathcal{I}$ and $A\in 2^X.$
$$\begin{array}{rcl} 
(\mathcal{J}\subseteq \mathcal{I})(A\in 2^X)&\Rightarrow & (\forall J\in\mathcal{J})(J\in \mathcal{I})(A\cap J\subseteq J) \\ &\Rightarrow & (\forall J\in\mathcal{J})(A\cap J \in \mathcal{I}) \\ &\Rightarrow & A\in (\mathcal{I}:\mathcal{J}).\end{array}$$
 
$(\Leftarrow):$ Let $(\mathcal{I}:\mathcal{J})=2^X$ and $J\in \mathcal{J}.$ We will show that $J\in\mathcal{I}$.
$$\left. \begin{array}{rr} 
 (\mathcal{I} :\mathcal{J})=2^X\Rightarrow X\in (\mathcal{I} :\mathcal{J}) \\ J\in\mathcal{J} \end{array}	\right\}\Rightarrow J=X\cap J \in \mathcal{I}.$$

$c)$ Let $\mathcal{J}\subseteq \mathcal{J}'$ and $A\in (\mathcal{I}:\mathcal{J}').$ We will show that $A\in (\mathcal{I}:\mathcal{J}).$
$$\left. \begin{array}{rr} 
A\in (\mathcal{I}:\mathcal{J}') \Rightarrow (\forall J \in \mathcal{J}' ) (A\cap J \in \mathcal{I}) \\ \mathcal{J}\subseteq \mathcal{J}' \end{array}	\right\}
  \Rightarrow (\forall J \in \mathcal{J} ) (A\cap J \in \mathcal{I})\Rightarrow A\in (\mathcal{I}:\mathcal{J}).$$

$d)$ Let $X\in \mathcal{J}$ and $A\in (\mathcal{I} :\mathcal{J}).$ We will show that $A\in \mathcal{I}.$
\\

$\left.\begin{array}{rr} A\in (\mathcal{I} :\mathcal{J})\Rightarrow (\forall J\in \mathcal{J})(A\cap J\in \mathcal{I}) \\ X\in \mathcal{J}\end{array}\right\}\Rightarrow A=A\cap X \in \mathcal{I}.$
\\

$e)$ Let $A \subseteq X.$
$$\begin{array}{rcl} A \in (\mathcal{I}\cap \mathcal{I}':\mathcal{J})& \Leftrightarrow &(\forall J \in \mathcal{J})(A\cap J \in \mathcal{I}\cap \mathcal{I}') \\ & \Leftrightarrow & (\forall J \in \mathcal{J}) (A\cap J \in \mathcal{I}\wedge A\cap J \in \mathcal{I}') \\ & \Leftrightarrow &(\forall J \in \mathcal{J})(A\cap J \in \mathcal{I})\wedge (\forall J \in \mathcal{J}) (A\cap J \in \mathcal{I}')   \\ & \Leftrightarrow & A \in (\mathcal{I}:\mathcal{J})\wedge A\in (\mathcal{I}':\mathcal{J}) \\ & \Leftrightarrow & A \in (\mathcal{I}:\mathcal{J})\cap(\mathcal{I}':\mathcal{J}) . \qedhere
\end{array}$$
\end{proof}

\begin{remark}
The converse of the implications given in Theorem \ref{1} $(c),(d)$ need not always to be true as shown by the following examples. 
\end{remark}

\begin{example}
Let $X=\{a,b,c\}$ and $\mathcal{I}=\{\emptyset,\{a\}\}.$ For the subfamilies $\mathcal{J}=\{\{a\},\{a,c\}\}$ and $\mathcal{J}'=\{\{a\},\{a,b\},\{a,c\}\}$ of $2^X,$ we have $(\mathcal{I}:\mathcal{J})=\{\emptyset,\{a\},\{b\},\{a,b\}\}$ and $(\mathcal{I}:\mathcal{J}')=\{\emptyset,\{a\}\}.$
\\

$a)$ $(\mathcal{I}:\mathcal{J}')\subseteq (\mathcal{I}:\mathcal{J}),$ but $\mathcal{J}\nsubseteq \mathcal{J}';$
\\

$b)$ $(\mathcal{I}:\mathcal{J}')=\mathcal{I},$ but $X\notin \mathcal{J}.$
\end{example}

\begin{corollary}
Let $\mathcal{I}$ be an ideal on $X$ and $\tau \subseteq 2^X.$ Then, the following statements hold. \\

  a) $(\mathcal{I}:\mathcal{I}) = (2^X:\mathcal{I}) = (\mathcal{I}:\{\emptyset\} ) = 2^X;$\\
  
  b) if $\tau$ is a topology on $X,$ then  $(\mathcal{I}:\tau) = \mathcal{I}.$
 \end{corollary}

\begin{remark}  Let $\mathcal{I}$ be an ideal on $X$ and $\mathcal{J} \subseteq 2^X.$ Then, $\mathcal{I}$ and $(\mathcal{I}:\mathcal{J})$ need not to be equal as shown by the following example. 
\end{remark} 
\begin{example}
Let $X=\{a,b,c\},$  $\mathcal{J}=\{\{a\}, \{a,c\} \}$ and $\mathcal{I}=\{\emptyset,\{a\}\}.$ Then, simple calculations show that $(\mathcal{I}:\mathcal{J})=\{\emptyset , \{a\}, \{b\},\{a,b\}\}$ which is not equal to $\mathcal{I}$.
\end{example}

\begin{theorem}\label{2}
Let $\mathcal{I}$ be a proper ideal on $X.$ Then,
$\mathcal{I}$ is a maximal ideal if and only if $(\mathcal{I}:\mathcal{J}) = \mathcal{I}$ for each $\mathcal{J}\subseteq 2^X$ with the property $\mathcal{J}\nsubseteq \mathcal{I}.$ 
\end{theorem}

\begin{proof}
  $(\Rightarrow):$ Let $\mathcal{I} $ be a maximal ideal and $\mathcal{J}\subseteq 2^X $. Let $ \mathcal{J}\nsubseteq \mathcal{I}$ and $A \in (\mathcal{I}:\mathcal{J}).$ We will prove that $A\in \mathcal{I}.$
  \\
  
 $\left. \begin{array}{rr} 
 A\in (\mathcal{I}:\mathcal{J})\Rightarrow (\forall J \in \mathcal{J}) (A\cap J \in \mathcal{I}) \\ \mathcal{I} \text{ is a maximal ideal} \end{array}	\right\}\Rightarrow (\forall J\in\mathcal{J})(A\in \mathcal{I} \vee J\in \mathcal{I})$
\\
$\left. \begin{array}{r} 
\Rightarrow A\in \mathcal{I} \vee (\forall J\in\mathcal{J})(J\in \mathcal{I}) \Rightarrow  (A\in \mathcal{I}\vee \mathcal{J}\subseteq \mathcal{I}) \\ \mathcal{J}\nsubseteq \mathcal{I} \end{array}	\right\} \Rightarrow  A\in \mathcal{I}.$
\\

$(\Leftarrow):$ Suppose that $\mathcal{I} $ is not a maximal ideal on $X.$
\\

 $\left. \begin{array}{rr} 
 \mathcal{I} \text{ is not a maximal ideal}\Rightarrow (\exists A\subseteq X)(A\notin \mathcal{I} \wedge X\setminus A \notin \mathcal{I} )\\ \mathcal{J}:=\{X\setminus A\}  \end{array}	\right\}\Rightarrow $
\\
$
\begin{array}{l}
\Rightarrow  (A\notin \mathcal{I})(\mathcal{J}\nsubseteq \mathcal{I})(A\in (\mathcal{I} :\mathcal{J}))
\end{array}
$
\\
$
\begin{array}{l}
\Rightarrow  (\mathcal{J}\nsubseteq \mathcal{I})((\mathcal{I} :\mathcal{J})\neq\mathcal{I})
\end{array}
$

This contradicts by hypothesis.
\end{proof}

\begin{corollary}
Let $\mathcal{I}$ be a proper ideal on $X.$ Then, $\mathcal{I}$ is a maximal ideal if and only if $(\mathcal{I}:\mathcal{J}) =\mathcal{I}$ or $(\mathcal{I}:\mathcal{J}) =2^X$ for all $\mathcal{J}\subseteq 2^X.$
\end{corollary}

\begin{proof}
This follows from Theorems \ref{22} and \ref{1}.
\end{proof}

\section{Annihilator of a Set Family}

\begin{definition}
 Let $X$ be a non-empty set and $\mathcal{J}\subseteq 2^X$. If $\mathcal{I}=\{\emptyset\} $, then the ideal quotient $(\{\emptyset\}:\mathcal{J}) $ is called  the annihilator  of  $\mathcal{J}$ and denoted by $Ann(\mathcal{J}).$ The notation $Ann_A$ will be used to denote $Ann(\{A\}),$ where $A\subseteq X.$
\end{definition}
\begin{corollary}
Let $X$ be a non-empty set and $A\subseteq X$. It is not difficult to see that $Ann_A = Ann(\mathcal{I}(A)).$ 
\end{corollary}
\begin{definition}
Let $X$ be a non-empty set and $\mathcal{J}\subseteq 2^X$. Then, $\mathcal{J}$ is said to be faithful if $Ann(\mathcal{J})=\{\emptyset\}.$
\end{definition}
\begin{corollary}
 Every topology on $X$ is faithful.
\end{corollary}

\begin{lemma}\label{5}
Let $\mathcal{I}$ be an ideal on $X.$ Then, $\mathcal{I}\cap Ann(\mathcal{I})=\{\emptyset\}$.
\end{lemma}
\begin{proof}
 Let $A\in \mathcal{I} \cap Ann(\mathcal{I}).$ We will show that $A=\emptyset.$
\\

$\begin{array}{l} A\in \mathcal{I} \cap Ann(\mathcal{I}) \Rightarrow (A\in \mathcal{I})(A\in Ann(\mathcal{I})) \Rightarrow A\cap A  =\emptyset \Rightarrow A=\emptyset.
\end{array}$
\end{proof}

\begin{theorem}
Let $X$ be an infinite set. Then, the family of all finite subsets of $X,$ denoted by  $\mathcal{I}_f,$ is a faithful ideal on $X$.
\end{theorem}

\begin{proof}
Suppose that $Ann(\mathcal{I}_f)\neq \{\emptyset\}.$
$$\begin{array}{rcl}Ann(\mathcal{I}_f)\neq \{\emptyset\} & \Rightarrow & (\exists A\in Ann(\mathcal{I}_f))(A\neq \emptyset) \\ & \Rightarrow & (\exists x\in X)(x\in A) \\ & \Rightarrow & (\{x\}\subseteq A)(|\{x\}|=1<\aleph_0) \\ & \Rightarrow & (\{x\}\in Ann(\mathcal{I}_f))(\{x\}\in\mathcal{I}_f) \\ & \Rightarrow & \{x\}\in \mathcal{I}_f \cap Ann(\mathcal{I}_f)  
\end{array}$$

This contradicts with Lemma \ref{5}.
\end{proof}
\begin{theorem}
Let $\mathcal{I}$ be a  faithful ideal  on $X.$ Then, $Ann(Ann(\mathcal{I}))=2^X.$
\end{theorem}
\begin{proof} Let $\mathcal{I}$ be a faithful ideal.\\
     $\left.\begin{array}{rr} \mathcal{I} \text{ is  faithful} \Rightarrow  Ann(\mathcal{I})= \{\emptyset\} \Rightarrow  Ann(Ann(\mathcal{I}))=Ann(\{\emptyset\} )=(\{\emptyset\}:\{\emptyset\}) \\ \text{Theorem }\ref{1}(b)
\end{array}\right\}\Rightarrow $
\\
$\begin{array}{l}
\Rightarrow  Ann(Ann(\mathcal{I}))=2^X.   
\end{array}$
\end{proof}


\begin{theorem}
If $\mathcal{I}$ is not a faithful ideal on $X$, then $Ann(Ann(\mathcal{I}))=\mathcal{I}.$
\end{theorem}
\begin{proof}
Let $A\notin \mathcal{I}.$ There are two cases. Now, we will process them.
\\

\textit{First case}: Let $A\cap I=\emptyset$ for all $I\in\mathcal{I}.$
\\
$\left.\begin{array}{rr}(\forall I\in\mathcal{I})(A\cap I=\emptyset)\Rightarrow A\in Ann(\mathcal{I}) \\ A\notin \mathcal{I}\Rightarrow A\neq \emptyset \end{array}\right\}\overset{\text{Lemma } \ref{5}}{\Rightarrow} A\notin Ann(Ann(\mathcal{I}))\ldots (1)$
\\

\textit{Second case}: Suppose that there exists $I\in\mathcal{I}$ such that $A\cap I\neq \emptyset.$ \\ Now, let $T:=\bigcap\{A\setminus I|I\cap A\neq \emptyset\}.$ 
\\
$\begin{array}{rcl} T:=\bigcap\{A\setminus I|I\cap A\neq \emptyset\}&\Rightarrow& (\emptyset\neq T\subseteq A)(\forall I\in\mathcal{I})(I\cap T=\emptyset)\\ &\Rightarrow& (\emptyset\neq T\subseteq A)(T\in Ann (\mathcal{I})) \\ &\overset{\text{Lemma } \ref{5}}{\Rightarrow} & (\emptyset\neq T\subseteq A)(T\notin Ann(Ann (\mathcal{I}))) \\ &\Rightarrow & A\notin Ann(Ann (\mathcal{I}))\ldots (2)\end{array} $
\\

Then, we have $Ann (Ann(\mathcal{I}))\subseteq \mathcal{I}$ from (1) and (2).
\\

Now, let $A \in \mathcal{I}$ and $J\in Ann(\mathcal{I})$. We will show that $A\in Ann(Ann(\mathcal{I})).$
\\

$\left.\begin{array}{r}
J\in Ann(\mathcal{I}) \Rightarrow (\forall I\in\mathcal{I})(I\cap J=\emptyset) \\ A\in\mathcal{I}  
\end{array}\right\}\Rightarrow A\cap J=\emptyset$
\\

Then, we have $A\in Ann(Ann(\mathcal{I}))$ and so $\mathcal{I} \subseteq Ann(Ann(\mathcal{I}))\ldots (3)$

$
(2),(3)\Rightarrow Ann(Ann(\mathcal{I}))=\mathcal{I}.
$
\end{proof}

\begin{theorem}
Let $\mathcal{I}$ be an ideal on $X$ and $A \subseteq X.$ Then, $Ann(\mathcal{I}(A))= Ann_A=\mathcal{I}_\epsilon(A).$
\end{theorem}
\begin{proof}
It is obvious that $Ann_A=\mathcal{I}_{\epsilon}(A).$ Also, we have
$$\begin{array}{rcl} Ann(\mathcal{I}(A)) & = &  \{J | (\forall I \in \mathcal{I}(A) ) (J \cap I =\emptyset )\} \\ & = & \{J | (\forall I\subseteq A) ) (J \cap I =\emptyset )\} \\ & = & \{J |  J \cap A =\emptyset\} \\  & = & \mathcal{I}_\epsilon(A). \qedhere
\end{array}$$
\end{proof}
\begin{corollary} \label{MaxAnn}
Let $\mathcal{I}$ be an ideal on $X.$ Then, $\mathcal{I}$ is a minimal ideal if and only if $Ann(\mathcal{I})$ is a maximal ideal on $X.$ 
\end{corollary}

\begin{proof}
This follows from Theorem \ref{17}.    
\end{proof}

\begin{theorem}\label{dens}
 Let  $\left(X,\tau,\mathcal{I}\right)$ be an ideal topological space and $A\subseteq X.$ Then, $A$ is a dense set in $X$ if and only if $\left(X,\tau,Ann_A\right)$ is a Hayashi-Samuel space.
\end{theorem}
\begin{proof}
$(\Rightarrow ): $ Let $A$ be a dense set in $X$ and $J\in \tau \cap Ann_A$. We will show that $J=\emptyset.$
\\ 
$\left. \begin{array}{rr} 
J\in \tau \cap Ann_A \Rightarrow (J\in \tau)(J\in Ann_A)\Rightarrow (J\in \tau)(A\cap J = \emptyset) \\ A \text{ is dense in } X\Rightarrow cl(A)= X\Rightarrow (\forall x\in X)(\forall U\in\tau(x))(U\cap A\neq\emptyset) \end{array}	\right\}
\!\Rightarrow\! J = \emptyset.$ 
\\ 

$(\Leftarrow ): $ Let $(X,\tau,Ann_A)$ be a Hayashi-Samuel space. We will show that $cl(A)=X.$ Let $x\in X $ and $U\in \tau(x).$ We will show that $U\cap A\neq\emptyset.$
  \\ 
  $\left. \begin{array}{rr} 
   (x\in X)(U \in\tau(x)) \Rightarrow  U\in \tau\setminus \{\emptyset\} \\ (X,\tau,Ann_A) \text{ is a Hayashi-Samuel space}\Rightarrow \tau \cap Ann_A = \{\emptyset\}\end{array}	\right\}
   \Rightarrow U\notin Ann_A$
   \\
   $\begin{array}{l}\Rightarrow U\cap A\neq \emptyset\end{array}$
     
   Then, we have $x\in cl(A)$ and so $X\subseteq cl(A).$ On the other hand, we have always $cl(A)\subseteq X.$ Thus, $cl(A)=X$ i.e. $A$ is a dense set in $X.$
\end{proof}

\section{Sharp Operator and Sharp Topology}

\begin{definition}
Let $\left(X,\tau,\mathcal{I}\right)$ be an ideal topological space. Then any subset $A$ of $X,$ $A^{\sharp}\left(\mathcal{I},\tau\right):=\{ x\in X|\left(\forall U\in \tau\left(x\right)\right)(\exists I\in \mathcal{I}\setminus\{\emptyset\}) \left(I\cap (U\cap A)^c=\emptyset\right)\}$ is called the sharp function of $A$ with respect to $\mathcal{I}$ and $\tau.$ If there is no ambiguity, we will write $A^{\sharp}(\mathcal{I})$ or simply $A^{\sharp}$ for  $A^{\sharp}(\mathcal{I},\tau).$
\end{definition}

\begin{theorem}\label{7}
Let  $\left(X,\tau,\mathcal{I}\right)$ be an ideal topological space and $A\subseteq X.$ Then, $A^{\sharp}\left(\mathcal{I},\tau\right)= A^{*}\left(Ann(\mathcal{I}),\tau\right).$ 
\end{theorem}

\begin{proof}
Let $A\subseteq X.$
$$\begin{array}{rcl}x \in  A^{\sharp}\left(\mathcal{I},\tau\right) & \Leftrightarrow & (\forall U\in \tau(x)) (\exists I\in \mathcal{I}\setminus\{\emptyset\} ) (I\cap (U\cap A)^c=\emptyset )\\ & \Leftrightarrow & (\forall U\in \tau(x)) (\exists I\in \mathcal{I}\setminus\{\emptyset\} ) (I\subseteq U\cap A )\\ & \Leftrightarrow & (\forall U\in \tau(x)) (\exists I\in \mathcal{I}\setminus\{\emptyset\} ) ( I\cap U\cap A \neq \emptyset ) \\ & \Leftrightarrow & (\forall U\in \tau(x))  ( U\cap A \notin Ann(\mathcal{I}))  \\ & \Leftrightarrow & x\in A^{*}\left(Ann(\mathcal{I}),\tau\right). \qedhere
\end{array}$$
\end{proof}

\begin{theorem}
 Let  $\left(X,\tau,\mathcal{I}\right)$ be an ideal topological space. Then, the following statements hold:
 \\

    $a)$ $A\subseteq B \Rightarrow A^\sharp \subseteq B^\sharp ,$
    \\
    
   $b)$ $A^\sharp = cl(A^\sharp) \subseteq cl(A),$
    \\ 
    
    $c)$ $(A\cap B)^\sharp \subseteq A^\sharp \cap B^\sharp ,$
    \\
    
    $d)$ $(A\cup B)^\sharp = A^\sharp \cup B^\sharp ,$
    \\
    
    $e)$ \ $A^\sharp \setminus B^\sharp \subseteq (A\setminus B)^\sharp ,$
    \\
    
    $f)$ $A\in Ann(\mathcal{I}) \Rightarrow A^\sharp = \emptyset,$
    \\
    
    $g)$ $A \in Ann(\mathcal{I}) \Rightarrow (A\cup B)^\sharp = B^\sharp = (A\setminus B )^\sharp,$
    \\
    
    $h)$ if $\mathcal{I} \text{ is faithful, then } A^\sharp = cl(A).$
\end{theorem}

\begin{proof}
This follows from the properties of local function and Theorem \ref{7}.
\end{proof}

\begin{theorem} \label{sharpstar} Let  $\left(X,\tau,\mathcal{I}\right)$ be an ideal topological space and $A\subseteq X.$ Then, 
$A^\sharp (\mathcal{I},\tau) \cup A^*(\mathcal{I},\tau)=cl(A).$
\end{theorem}
\begin{proof}
Let $A\subseteq X.$ 
\\
$\left.\begin{array}{rr}
A\subseteq X\Rightarrow A^{\sharp}(\mathcal{I},\tau)\subseteq cl(A) \\  
A\subseteq X\Rightarrow A^*(\mathcal{I},\tau)\subseteq cl(A)\end{array}\right\}\Rightarrow A^{\sharp}(\mathcal{I},\tau)\cup A^*(\mathcal{I},\tau)\subseteq cl(A)\ldots (1)$
\\

Now, let $x \notin  A^{\sharp}\left(\mathcal{I},\tau\right) \cup A^*\left(\mathcal{I},\tau\right).$
\\
$\begin{array}{l}x \notin  A^{\sharp}\left(\mathcal{I},\tau\right) \cup A^*\left(\mathcal{I},\tau\right)\Rightarrow (x \notin  A^{\sharp}\left(\mathcal{I},\tau\right))(x \notin  A^*\left(\mathcal{I},\tau\right))
\end{array}$
\\
$\left.\begin{array}{rcl}
\Rightarrow (\exists U\in \tau(x))(U\cap A \in \mathcal{I}) (\exists V\in\tau(x))(V\cap A \in Ann(\mathcal{I})) \\ W:= U\cap V  
\end{array}\right\}\ \Rightarrow 
$
\\ 
$\begin{array}{l} \Rightarrow (W \in \tau(x))(W\cap A \in \mathcal{I}) (W\cap A \in Ann(\mathcal{I})) 
\end{array}$
\\ 
$\begin{array}{l}  \Rightarrow ( W\in \tau(x))  ( W\cap A \in \mathcal{I} \cap  Ann(\mathcal{I}) )  
\end{array}$
\\ 
$\begin{array}{l}   \overset{\text{Lemma }\ref{5}}{\Rightarrow} ( W\in \tau(x))  ( W\cap A \in \{\emptyset\} ) 
\end{array}$
\\ 
$\begin{array}{l}   \Rightarrow ( W\in \tau(x))  ( W\cap A = \emptyset ) 
\end{array}$
\\ 
$\begin{array}{l}  \Rightarrow  x\notin cl(A)
\end{array}$

Then, we have $cl(A)\subseteq A^{\sharp}\left(\mathcal{I},\tau\right) \cup A^*\left(\mathcal{I},\tau\right)\ldots (2)$

$(1),(2)\Rightarrow A^{\sharp}\left(\mathcal{I},\tau\right) \cup A^*\left(\mathcal{I},\tau\right)=cl(A).$
\end{proof}

\begin{corollary}\label{33}
  Let  $\left(X,\tau,\mathcal{I}\right)$ be an ideal topological space and $A\subseteq X.$ Then, the following properties hold:\\
  
  $a)$ If $A\in \mathcal{I},$ then $A^\sharp = cl(A);$ \\

   $b)$ If $A\in Ann(\mathcal{I}),$ then $A^*=cl(A);$ \\

    $c)$ If $A^\sharp = \emptyset,$ then $A^* = cl(A);$ \\

    $d)$ If $A^* = \emptyset,$ then $A^\sharp = cl(A).$ \\
 
\end{corollary}


\begin{definition}
Let $(X,\tau,\mathcal{I})$ be an ideal topological space. We consider a map $cl^{\sharp}: 2^X\rightarrow 2^X$ as $cl^{\sharp}(A)=A\cup A^{\sharp}$, where $A$ is any subset of $X$. 
\end{definition}

\begin{theorem} \label{2}
Let $(X,\tau,\mathcal{I})$ be an ideal topological space and $A,B\subseteq X.$ Then, the following statements hold: \\

$a)$ $cl^{\sharp}(\emptyset)= \emptyset$,\\

$b)$ $cl^{\sharp}(X)= X$,\\

$c)$ $A\subseteq cl^{\sharp}(A)$,\\

$d)$ If $A\subseteq B$, then $cl^{\sharp}(A)\subseteq cl^{\sharp}(B)$,\\

$e)$ $cl^{\sharp}(A)\cup cl^{\sharp}(B)= cl^{\sharp}(A\cup B),$\\

$f)$ $cl^{\sharp}(cl^{\sharp}(A))=cl^{\sharp}(A).$
\\
\end{theorem}

\begin{proof}
Let $A,B\subseteq X.$\\

$a)$ Since $\emptyset^{\sharp}=\emptyset,$ we have $cl^{\sharp}(\emptyset)=\emptyset\cup \emptyset^{\sharp}=\emptyset.$  
\\

$b)$ Since $X\cup X^{\sharp}=X,$ we have $cl^{\sharp}(X)=X.$\\

$c)$ Since $cl^{\sharp}(A)=A\cup A^{\sharp},$ we have $A\subseteq cl^{\sharp}(A).$\\

$d)$ Let $A\subseteq B.$ We get from Theorem \ref{1}$(e)$ that $A^{\sharp}\subseteq B^{\sharp}.$ Therefore, we have
$A\cup A^{\sharp}\subseteq B\cup B^{\sharp}$ which means that $cl^{\sharp}(A)\subseteq cl^{\sharp}(B).$
\\

$e)$ This follows from the definition of operator $cl^{\sharp}$ and Theorem \ref{1}$(e).$ 
\\

$f)$ This follows from $(c)$ that $cl^{\sharp}(A)\subseteq cl^{\sharp}(cl^{\sharp}(A)).$ On the other hand, since $A^{\sharp}$ is closed in $X,$ we have  $(A^{\sharp})^{\sharp}\subseteq A^{\sharp}.$ Therefore,

$$\begin{array}{rcl}
cl^{\sharp}(cl^{\sharp}(A)) & = & cl^{\sharp}(A)\cup \left(cl^{\sharp}(A)\right)^{\sharp} \\ & = & cl^{\sharp}(A)\cup (A\cup A^{\sharp})^{\sharp} \\ & = & cl^{\sharp}(A)\cup A^{\sharp}\cup (A^{\sharp})^{\sharp}
\\ & \subseteq  & cl^{\sharp}(A)\cup A^{\sharp}\cup A^{\sharp} \\ & = & cl^{\sharp}(A)
\end{array}$$

Thus, we have $cl^{\sharp}(cl^{\sharp}(A))=cl^{\sharp}(A).$
\end{proof}

\begin{corollary}
Let $(X,\tau,\mathcal{I})$ be an ideal topological space. Then, the function $cl^{\sharp}: 2^X\rightarrow 2^X$ defined by $cl^{\sharp}(A)=A\cup A^{\sharp}$, where $A$ is any subset of $X,$ is a Kuratowski closure operator.
\end{corollary}

\begin{definition}
Let $(X,\tau,\mathcal{I})$ be an ideal topological space. Then, the family $\tau^{\sharp}=\{A\subseteq X|cl^{\sharp}(X\setminus A)=X\setminus A\}$ is a topology called $\sharp$-topology on $X$ induced by topology $\tau$ and ideal $\mathcal{I}.$  We can also write $\tau_{\mathcal{I}}^{\sharp}$ instead of $\tau^{\sharp}$ to specify the ideal as per our requirements.
\end{definition}

\begin{remark}
We have the following diagram from the definitions of $*$-topology and $\sharp$-topology. The following example shows that these implications are not reversible. Also, the notions of $\tau^*$-open set and $\tau^{\sharp}$-open set are independent.
$$\begin{array}{cccc}
 & \tau^*\text{-open}  &  & \tau^{\sharp}\text{-open}   \\
 & & \nwarrow \mbox{ } \mbox{ }\nearrow &     \\
 &  & \tau\text{-open} &   
\end{array}$$
\end{remark}

\begin{example}
Let $X=\{a,b,c\},$ $\tau=\{\emptyset,X,\{a,c\}\}$ ve $\mathcal{I}=\{\emptyset,\{a\},\{b\},\{a,b\}\}.$ Simple calculations show that $Ann(\mathcal{I})=\{\emptyset,\{c\}\},$ $\tau^*=\{\emptyset,X,\{a,c\},\{b,c\},\{c\}\}$ and $\tau^{\sharp}=\{\emptyset,X,\{a,b\},\{a,c\},\{a\}\}.$
\\

$a)$ The set $\{c\}\in\tau^*,$ but $\{c\}\notin \tau;$
\\

$b)$ The set $\{a\}\in\tau^{\sharp},$ but $\{a\}\notin \tau;$
\\

$c)$ The set $\{c\}\in \tau^*,$ but $\{c\}\notin\tau^{\sharp};$
\\

$d)$ The set $\{a\}\in\tau^{\sharp},$ but $\{a\}\notin\tau^*.$
\end{example}
\begin{definition}
Let $\left(X,\tau,\mathcal{I} \right )$ be an ideal topological space. We define the operator $\Psi^\sharp:2^X\rightarrow 2^X$ as $\Psi^\sharp(A)= X\setminus (X\setminus A)^\sharp$ for any subset $A$ of $X.$ We can also write $\Psi^\sharp (A(\mathcal{I},\tau))$ instead of $\Psi^\sharp (A)$ to specify the ideal and the topology as per our requirements.
\end{definition}

\begin{corollary}
Let  $\left(X,\tau,\mathcal{I}\right)$ be an ideal topological space and $A\subseteq X.$ Then, $\Psi\left(A\left(Ann(\mathcal{I}),\tau\right)\right)= \Psi^{\sharp}\left(A\left(\mathcal{I},\tau\right)\right).$  
\end{corollary}

\begin{proof}
This follows from the definition of $\Psi$-operator and Theorem \ref{7}.
\end{proof}
\begin{theorem}
Let $\left(X,\tau,\mathcal{I} \right )$ be an ideal topological space and $A \subseteq X.$ Then, $A$ is $\tau^{\sharp}$-open if and only if $A \subseteq \Psi^\sharp(A).$
\end{theorem}
\begin{proof}
Let $A\subseteq X.$
$$\begin{array}{rcl} A\in\tau^{\sharp} & \Leftrightarrow  & cl^{\sharp}(X\setminus A)= (X\setminus A) \\ & \Leftrightarrow & (X\setminus A)\cup (X\setminus A)^{\sharp} =X\setminus A \\ & \Leftrightarrow & (X\setminus A)^{\sharp}\subseteq X\setminus A \\ & \Leftrightarrow & A\subseteq X\setminus (X\setminus A)^\sharp \\ & \Leftrightarrow & A\subseteq \Psi^{\sharp}(A). \qedhere
\end{array}$$
\end{proof}
\begin{theorem}
 Let  $\left(X,\tau,\mathcal{I}\right)$ be an ideal topological space and $A\subseteq X.$ Then, $\Psi^\sharp(A) \cap \Psi(A) = int(A).$
\end{theorem}
\begin{proof}
 Let $A\subseteq X.$
$$\begin{array}{rcl} \Psi^\sharp(A) \cap \Psi(A) & = & [X\setminus (X\setminus A)^\sharp] \cap [X\setminus (X\setminus A)^*] \\ & = & X\setminus [(X\setminus A)^\sharp \cup (X\setminus A)^*]\\ & \overset{\text{Theorem }\ref{sharpstar}}{=} & X\setminus cl(X\setminus A) \\ & = & X\setminus (X\setminus int(A)) \\ & = & int(A). \qedhere
\end{array}$$
\end{proof}

\begin{corollary}
  Let  $\left(X,\tau,\mathcal{I}\right)$ be an ideal topological space and $A\subseteq X.$
  \\
  
  $a)$ If $X\setminus A\in \mathcal{I},$ then $\Psi^\sharp(A) = int(A);$
  \\

   $b)$ If $A\in Ann(\mathcal{I}),$ then $\Psi(A) = int(A).$
\end{corollary}

\begin{theorem} \label{decopen}
Let $\left(X,\tau,\mathcal{I} \right )$ be an ideal topological space and $A\subseteq X.$ Then, $A$ is $\tau$-open if and only if $A$ is both $\tau^*$-open and $\tau^{\sharp}$-open.
\end{theorem}
\begin{proof} It is clear from the equality below.

$$\begin{array}{rcl}\tau &=&\tau^*(\{\emptyset\},\tau)\\ &=&\tau^*(\mathcal{I}\cap Ann(\mathcal{I}),\tau)\\ &=&\tau^*(\mathcal{I},\tau)\cap \tau^*(Ann(\mathcal{I}),\tau)\\ &=&\tau^*(\mathcal{I},\tau)\cap \tau^\sharp (\mathcal{I},\tau) \\ &=&\tau^*\cap \tau^\sharp.\qedhere \end{array}$$
\end{proof}

\begin{theorem}
Let $\mathcal{I}$ be a proper ideal on $X.$ If $\mathcal{I}$ is a minimal ideal, then $A^\sharp=\emptyset$ or $(X\setminus A)^\sharp=\emptyset$ for all $A\subseteq X.$\end{theorem}

\begin{proof}
Let $A\subseteq X.$\\
$\left. \begin{array}{rr} 
 \mathcal{I} \text{ is minimal ideal}\\ \text{Corollary } \ref{MaxAnn} \end{array}	\right\}\Rightarrow \!\!\!\!\! \begin{array}{c} \\
\left. \begin{array}{r} 
 Ann(\mathcal{I}) \text{ is maximal ideal} \\  A\subseteq X\end{array}	\right\} \Rightarrow\end{array}$
\\
$\begin{array}{l}\Rightarrow A\in Ann(\mathcal{I})\vee (X\setminus A)\in Ann(\mathcal{I})
\end{array}$
\\
$\begin{array}{l}\Rightarrow A^{\sharp}=\emptyset \vee (X\setminus A)^{\sharp}=\emptyset.
\end{array}$
\end{proof}

\begin{corollary}
Let $\mathcal{I}$ be a minimal ideal on $X$ and $A\subseteq X.$ Then, $A$ is $\tau^\sharp$-closed or $\tau^\sharp$-open.
\end{corollary}



\begin{theorem}\label{maxsharp}
Let $\mathcal{I}$ be a proper ideal on $X.$ If $\mathcal{I}$ is a maximal ideal, then $A^\sharp= cl(A)$ or $\Psi^\sharp(A)=int(A)$ for all $A\subseteq X.$
\end{theorem}

\begin{proof}
Let $A\subseteq X.$ 
\\
$\left. \begin{array}{rr} 
 \mathcal{I} \text{ is maximal ideal} \\ A\subseteq X \end{array}	\right\}\Rightarrow \!\!\!\!\! \begin{array}{c} \\
\left. \begin{array}{r} 
 A\in \mathcal{I} \vee (X\setminus A)\in\mathcal{I} \Rightarrow A^*=\emptyset \vee (X\setminus A)^*=\emptyset \\  A\subseteq X\overset{\text{Theorem \ref{sharpstar}}}\Rightarrow cl(A)=A^*\cup A^\sharp\end{array}	\right\} \Rightarrow\end{array}$
\\
$
\begin{array}{l}
\Rightarrow  cl(A)=A^\sharp \vee X\setminus int(A)=cl(X\setminus A)=(X\setminus A)^\sharp \end{array}
$
\\
$
\begin{array}{l}
\Rightarrow cl(A)=A^\sharp \vee int(A)=X\setminus (X\setminus A)^\sharp =\Psi^\sharp (A).
\end{array}
$
\end{proof}

\begin{corollary}
  Let $\mathcal{I}$ be a proper ideal on $X.$ If $\mathcal{I}$ is a minimal ideal, then $A^*= cl(A)$ or $\Psi(A)=int(A)$ for all $A\subseteq X.$
\end{corollary}

\begin{proof}
This follows from Theorem  \ref{maxsharp}.
\end{proof}
    
    


\section{Decomposition of continuity}
\begin{definition}
A function $f:(X,\tau,\mathcal{I})\to (Y,\sigma)$ is called $*$-continuous ($\sharp$-continuous) if $f^{-1}[V] \in \tau^{*}$ $(f^{-1}[V] \in \tau^{\sharp})$ for each open set $V$ of $Y.$
\end{definition}
\begin{corollary}
    A function $f:(X, \tau, \mathcal{I})\to(Y, \sigma)$ is $*$-continuous if and only if $f:(X, \tau^*(\mathcal{I}))\to(Y, \sigma)$ is continuous.
\end{corollary}

\begin{corollary}
    A function $f:(X, \tau, \mathcal{I})\to(Y, \sigma)$ is $\sharp$-continuous if and only if $f:(X, \tau^\sharp(\mathcal{I}))\to(Y, \sigma)$ is continuous.
\end{corollary}

\begin{remark}
In \cite{jan}, the authors showed  that if
  $f:(X, \tau)\to(Y, \sigma)$ is a continuous 
function and $\mathcal{I}$ is ideal on $X$, then  $f:(X, \tau^*)\to(Y, \sigma)$  is also continuous. However, the converse need not always to be true as shown in \cite{jan}.
\end{remark}


\begin{corollary}\label{con}
 Let   $f:(X, \tau,\mathcal{I})\to(Y, \sigma)$ be a function. If $f$ is continuous, then it is also $\sharp$-continuous.
\end{corollary}

\begin{remark}
  The converse of Corollary \ref{con} need not to be true as shown by the following example.   
\end{remark}

\begin{example}
Let $\mathbb{R}$ be the real line with the usual topology $\mathcal{U},$ $\mathcal{I}=\{\emptyset,\{0\}\}$ and $\tau=\{A\subseteq \mathbb{R} | 0 \in A\}\cup\{\emptyset\}.$ Consider the identity function $i:\mathbb{R}\to \mathbb{R}.$ Now, let $\emptyset \neq A \in \tau$.
$$\begin{array}{rcl} \emptyset \neq A\in\tau & \Rightarrow  & 0\in A \\ & \Rightarrow & 0\notin X\setminus A\\ & \Rightarrow & X\setminus A \in Ann(\mathcal{I)} \\ & \Rightarrow & (X\setminus A)^\sharp =\emptyset \\ & \Rightarrow & A\subseteq X=X\setminus (X\setminus A)^\sharp=\Psi^\sharp(A) \\  & \Rightarrow & A \in \mathcal{U}^{\sharp}\end{array}$$ 
 Hence, $i:(\mathbb{R},\mathcal{U}^\sharp)\to (\mathbb{R},\tau)$ is continuous. However, $i:(\mathbb{R},\mathcal{U})\to (\mathbb{R},\tau)$ is not continuous since $A=\{0\}\in \tau$ but $\{0\}\notin \mathcal{U}.$
\end{example}

\begin{theorem}
Let $f:(X,\tau,\mathcal{I})\to (Y,\sigma)$ be a function. Then, $f$ is continuous if and only if $f$ is $*$-continuous and $\sharp$-continuous.
\end{theorem}
\begin{proof}
This follows from Theorem \ref{decopen}.
\end{proof}

\section{Some applications of sharp operator}
\begin{example}
By using the sharp topology, we prove that the set of all rational numbers $\mathbb{Q}$ is dense in $(\mathbb{R},\mathcal{U})$, where $\mathcal{U}$ is the usual topology on the set of all real numbers $\mathbb{R}$. For this, firstly we will prove that $(\mathbb{R} , \mathcal{U}, Ann(\mathbb{Q}))$ is Hayashi-Samuel space. This is obvious from the fact that $(a,b)\cap \mathbb{Q} \neq \emptyset$ for all $a,b \in \mathbb{R}$. Hence, by Theorem \ref{dens}, $\mathbb{Q}$ is dense in $\mathbb{R}$.
\end{example}

\begin{question}\label{ques}
 Let $\mathbb{R}$ be the set of all real numbers. Is there any Hausdorff space on $\mathbb{R}$ such that the set of all irrational numbers $\mathbb{I}$ is not dense, while the set of all rational numbers $\mathbb{Q}$ is dense? 
\end{question}

\begin{example}
 Let $\mathbb{R}$ be the set of all real numbers with the usual topology $\mathcal{U}$ and let $\mathcal{I}=\mathcal{I}(\mathbb{Q}),$ where $\mathbb{Q}$ is the set of all rational numbers.
 Now, let $\mathbb{I}$ be the set of all irrational numbers.
  Since $\mathbb{Q}\in\mathcal{I},$ by Lemma \ref{gulkız} and Corollary \ref{33}, we get $\mathbb{Q}^*=\emptyset$ and $\mathbb{Q}^{\sharp}=cl(\mathbb{Q})=\mathbb{R}.$ Thus, $cl^\sharp (\mathbb{Q})=\mathbb{Q}\cup \mathbb{Q}^\sharp=\mathbb{Q}\cup cl(\mathbb{Q})= \mathbb{R}.$ Hence, $\mathbb{Q}$ is a dense set in $(\mathbb{R},\mathcal{U}^\sharp).$ On the other hand, we have $\mathbb{I}^\sharp=\emptyset$ since $\mathbb{I}\in Ann(\mathcal{I}).$ Therefore, $cl^\sharp(\mathbb{I})= \mathbb{I}\cup \mathbb{I}^\sharp=\mathbb{I}.$  In other word, $\mathbb{I}$ is not a dense set in $(\mathbb{R},\mathcal{U}^\sharp).$ Finally, it is obvious that  $(\mathbb{R},\mathcal{U}^\sharp)$ is  Hausdorff since  $\mathcal{U}\subseteq \mathcal{U}^\sharp$ and $(\mathbb{R},\mathcal{U})$ is Hausdorff. 
\end{example}

\begin{question}\label{quesi}
 Let $\mathbb{R}$ be the set of all real numbers. Is there any Hausdorff space such that the set of all rational numbers $\mathbb{Q} $ is clopen? 
\end{question}

\begin{example}
Let $\mathbb{R}$ be the real line with the usual topology $\mathcal{U}$ and let $\mathcal{I}=\mathcal{I}(\{0\})=\{\emptyset,\{0\}\}$. Let $\mathbb{Q}$ be the set of all rational numbers and $\mathbb{I}$ be the set of all irrational numbers. By simple calculations, it's not difficult to see that if $0\in A,$ then $A^\sharp=\{0\}.$ Thus, we have $cl^\sharp(\mathbb{Q})=\mathbb{Q}\cup \mathbb{Q}^\sharp=\mathbb{Q}\cup \{0\}=\mathbb{Q},$ that is, $\mathbb{Q}$ is closed in $(\mathbb{R},\mathcal{U}^\sharp).$ On the other hand, it is easy to see that $Ann(\mathcal{I})=2^{\mathbb{R}\setminus\{0\}},$ where $2^{\mathbb{R}\setminus\{0\}}$ is the powerset of $\mathbb{R}\setminus\{0\}.$ Now, let $A\subseteq \mathbb{R}.$ 

\textit{First case:} Let $0\in A.$
$$0\in A \Rightarrow (A^\sharp=\{0\})(A^* = cl(A) \vee A^* = cl(A)\setminus \{0\}).$$

\textit{Second case:} Let $0\notin A.$
$$0\notin A \Rightarrow (A^\sharp= \emptyset)(A^* = cl(A)).$$

  Then, $cl^\sharp (\mathbb{Q}) =\mathbb{Q}$ and $cl^\sharp(\mathbb{I}) = \mathbb{I}.$ Thus, $\mathbb{Q}$ is clopen in $(\mathbb{R},\mathcal{U}^\sharp).$
\end{example}
\begin{lemma} \label{disc}
Let $(X,\tau)$ be a topological space. If $F\in C(X,\tau)\setminus \{\emptyset,X\}$, then the space $(X,\tau^\sharp, \mathcal{I}(F))$ is disconnected.
\end{lemma}
\begin{proof}
Let $F\in C(X,\tau)\setminus \{\emptyset,X\}.$ We will prove that $(X,\tau^\sharp, \mathcal{I}(F))$ is disconnected. It is sufficient to show that there exists a set which is clopen in $(X,\tau^\sharp, \mathcal{I}(F)).$
\\
$\left.\begin{array}{rr}
F\in C(X,\tau)\setminus \{\emptyset,X\}\Rightarrow F^c\in \tau\setminus \{\emptyset,X\}
      \\
  \tau\subseteq \tau^{\sharp}   
\end{array}\right\}\Rightarrow F^c\in \tau^{\sharp}\setminus \{\emptyset,X\}\ldots (1)$
\\
$\left.\begin{array}{rr} F^c\in Ann_F\Rightarrow (F^c)^{\sharp}=\emptyset \\ cl^{\sharp}(F^c)=(F^c)^{\sharp}\cup F^c \end{array}\right\}\Rightarrow cl^{\sharp}(F^c)= F^c\Rightarrow F^c\in C(X,\tau^{\sharp},\mathcal{I}(F))\setminus \{\emptyset,X\}\ldots (2)$
\\

$(1),(2)\Rightarrow F^c\in \left(\tau^{\sharp}\setminus \{\emptyset,X\}\right)\cap \left(C(X,\tau^{\sharp},\mathcal{I}(F))\setminus \{\emptyset,X\}\right).$
\end{proof}

\begin{example}
    In this example, we will build a disconnected Hausdorff space. Let $\mathbb{R}$ be the real line with the usual topology $\mathcal{U}$. Let $\mathcal{I}  = \mathcal{I}(\mathbb{N}),$ where $\mathbb{N}$ is the set of all natural numbers. The set $\mathbb{N}$ is closed in $(\mathbb{R},\mathcal{U})$. By Lemma \ref{disc}, $(\mathbb{R},\mathcal{U}^{\sharp},\mathcal{I}(\mathbb{N}) ) $ is disconnected.
\end{example}
{\bf Acknowledgements} We would like to thank the anonymous reviewers for their careful reading of our manuscript and their insightful comments and suggestions.
\\

{\bf Author Contributions} The authors contributed equally to this work. The
authors read and approved the final manuscript.
\\

{\bf Funding} Not applicable.
\\

{\bf Data availability} Enquiries about data availability should be directed to
the authors.
\\

{\bf Declarations}
\\

{\bf Ethical approval} This article does not contain any studies with human
participants or animals performed by any of the authors.
\\

{\bf Conflict of interest} Authors do not have any conflict of interest with
any other person or organization.
\\

{\bf Informed consent} Informed consent was obtained from all individual
participants included in the study.

\bibliographystyle{amsplain}

\end{document}